\documentclass{amsart}

\usepackage{amssymb}
\usepackage{amsmath}
\usepackage{amsfonts}
\usepackage{geometry}
\usepackage{bbm}
\usepackage{hyperref}
\usepackage{tikz}
\usepackage{mathrsfs}

\newcommand{\expn}{\operatorname{e}}

\usetikzlibrary{matrix,arrows,decorations.pathmorphing}

\tikzset{thick,level distance=3mm,sibling distance=6mm}

\newcounter{cprop}[section]

\newtheorem{theorem}[cprop]{Theorem}
\newtheorem*{theorem*}{Theorem}

\theoremstyle{plain}

\newtheorem*{corollary*}{Corollary}

\newtheorem{lemma}[cprop]{Lemma}
\newtheorem{proposition}[cprop]{Proposition}

\numberwithin{equation}{section}

\theoremstyle{definition}
\newtheorem{definition}[cprop]{Definition}
\newtheorem{example}[cprop]{Example}

\theoremstyle{remark}
\newtheorem{remark}[cprop]{Remark}

\newcommand{\R}{\mathbb{R}}

\newcommand{\N}{\mathbb{N}}
\newcommand{\Z}{\mathbb{Z}}

\newcommand{\vertiii}[1]{{\left\vert\kern-0.25ex\left\vert\kern-0.25ex\left\vert #1 
    \right\vert\kern-0.25ex\right\vert\kern-0.25ex\right\vert}}

\begin{document}
\title[Runge-Kutta for RDEs]{Runge-Kutta methods for rough differential equations}

\author{M. Redmann}
\address{Martin Redmann, Martin Luther University Halle-Wittenberg, Institute of Mathematics, Theodor-Lieser-Str. 5, 06120 Halle (Saale), Germany}
\email{martin.redmann@mathematik.uni-halle.de}

\author{S. Riedel}
\address{Sebastian Riedel \\
Institut f\"ur Mathematik, Technische Universit\"at Berlin, Germany and Weierstra{\ss}-Institut, Berlin, Germany}
\email{riedel@math.tu-berlin.de}

\renewcommand{\thefootnote}{\fnsymbol{footnote}} 
\footnotetext{2020 \emph{Mathematics Subject Classification.} 60H10, 60H35, 60L20, 60L70, 65C30.}     
\renewcommand{\thefootnote}{\arabic{footnote}} 

\keywords{B-series, rough paths, Runge-Kutta methods}

%
%

\begin{abstract}
  We study Runge-Kutta methods for rough differential equations which can be used to calculate solutions to stochastic differential equations driven by processes that are rougher than a Brownian motion. We use a Taylor series 
  representation (B-series) for both the numerical scheme and the solution of the rough differential equation in order to determine conditions that guarantee the desired order of the local error for the 
  underlying Runge-Kutta method. Subsequently, we prove the order of the global error given the local rate. In addition, we simplify the numerical approximation by introducing a Runge-Kutta scheme 
  that is based on the increments of the driver of the rough differential equation. This simplified method can be easily implemented and is computational cheap since it is derivative-free. 
  We provide a full characterization of this implementable Runge-Kutta method meaning that we provide necessary and sufficient algebraic conditions for an optimal order of convergence in case that the driver, e.g., is 
  a fractional Brownian motion with Hurst index $\frac{1}{4} < H \leq \frac{1}{2}$. We conclude this paper by conducting numerical experiments verifying the theoretical rate of convergence.
\end{abstract}

\maketitle

\section*{Introduction}

Ordinary differential equations (ODEs) have many real life applications. They, e.g., describe chemical, physiological and ecological processes or they appear as spatially discretized 
partial differential equations like the heat equation. Often analytic solutions to ODEs do not exist which requires numerical approximations in order to solve these equations. 
An important class of such schemes are Runge-Kutta methods \cite{butcher, hairer10sodI, hairer10sodII} which can be of arbitrary order of convergence.
These are often preferred in practice since they are derivative-free in contrast to Taylor methods. Computing derivatives of the right hand side function 
$f_0$ of an ODE can either be very costly or closed form expressions might not be available.\\
However, in many applications uncertainties need to be taken into account.
Therefore, for a more accurate modeling in such cases, a noise term can be added to an ODE leading to stochastic differential equations (SDEs). Runge-Kutta schemes for SDEs driven by a Brownian motion have already 
been established, see, e.g., \cite{BB00, debrabant08bao, kloeden99nso, milstein04snf, roessler10rkm}.\smallskip

Lyons' rough paths theory provides an alternative way to SDEs which goes far beyond the scope of usual It\=o equations. In this paper, we are interested in studying numerical schemes to solve 
rough differential equations (RDEs) of the form
\begin{align}\label{eqn:RDE_with_drift}
 dy(t) = f_0(y(t))\, dt + f(y(t))\, d\mathbf{X}(t),\quad y(t_0) = y_0 \in \R^m,
\end{align}
where $\mathbf{X}$ is a suitable rough path above some $\alpha$-H\"older path $X = (X^1,\ldots,X^d) \colon [0,T] \to \R^d$, 
$f = (f_1,\ldots,f_d)$ and $f_i \colon \R^m \to \R^m$ are vector fields for every $i = 0,\ldots,d$. Such equations represent SDEs driven by stochastic processes that are potentially rougher than a
Brownian motion if $\mathbf{X}$ is a random rough path, i.e. a stochastic process with sample paths lying in a rough paths space. One benefit of rough paths theory compared to It\=o's theory
is that one is not restricted to the martingale framework. In fact, there is a large class of stochastic processes which have ``natural extensions'' to rough paths valued processes, cf. \cite{FV10}.
For instance, many Gaussian processes possess such a ``natural lift'' including the fractional Brownian motion with Hurst parameter $H > 1/4$, but even more general processes like the bifractional Brownian motion,
Volterra processes or processes which can be represented by random Fourier series, cf. \cite{FGGR16} for a discussion.\smallskip

In the context of rough paths theory, numerical schemes are indispensable when simulating the solution to an RDE driven by a random rough path or when discretizing rough stochastic partial differential
equations \cite{BBRRS18}. In fact, numerical schemes played a fundamental role in rough paths theory from the very beginning. This is probably most visible in the work of Davie \cite{Dav07},
where the Milstein scheme is used to solve rough differential equations theoretically. This approach was generalized to higher order Taylor-type schemes by Friz and Victoir \cite{FV10}. However,
in a stochastic context, these schemes are of little use in practice since they contain iterated stochastic integrals whose distribution is unknown in general. 
To overcome this difficulty, Deya, Neuenkirch and Tindel introduced so-called \emph{simplified} schemes in \cite{DNT12} in which the iterated stochastic integrals are replaced by products of increments of the driving process.
These schemes were successfully used in different contexts, cf. e.g. \cite{BFRS16,BBRRS18}. 
However, as Taylor methods, these numerical approximations suffer from the need to calculate or simulate derivatives of the vector fields $f_k$. As mentioned above, even if the derivatives are available,
this can be very expensive especially in a large scale setting (e.g. spatially discretized rough partial differential equation). Moreover, the simplified scheme is difficult to implement in general. 
Therefore, we see the need of studying Runge-Kutta methods for rough differential equations that can easily be implemented and are derivative-free. \smallskip

Our approach to establish Runge-Kutta methods is classical, both in the deterministic and the stochastic context: First, we define a class of equations which can be expanded in a $B$-series. 
Second, we have to find a $B$-series representation of the equation \eqref{eqn:RDE_with_drift}. Comparing both series, we can, in principle, deduce the order conditions of the Runge-Kutta method by matching their coefficients.
A $B$-series representation of an ODE contains combinations of products and derivatives of the defining vector field which can be described in the language of trees. For SDEs, integrated products of iterated stochastic integrals
have to be considered in addition which can be described in the same language. We call such objects \emph{tree-iterated (stochastic) integrals} in the sequel. 
A rough path in the sense of Lyons \cite{Lyo98,LQ02,LCL07} is a collection of objects which ``mimic'' the iterated integrals of the underlying path. 
Lyons' theory is able to solve differential equations driven by \emph{geometric} rough paths, i.e. those which obey the usual chain rule. Gubinelli realized in \cite{Gub10} that one can even solve rough differential equations
driven by non-geometric rough paths if one additionally assumes that all tree-iterated integrals are known. He calls such objects \emph{branched rough paths}. Thinking of $B$-series representations of SDEs,
this is a very natural approach to solve equations of the form \eqref{eqn:RDE_with_drift}. For us, it is therefore reasonable to use his theory and to interpret the equation \eqref{eqn:RDE_with_drift} as a
rough differential equation driven by a branched rough path. Doing this, we are able to deduce the $B$-series expansion of \eqref{eqn:RDE_with_drift} in Theorem \ref{B-series_RDE}.
Comparing both $B$-series and matching their coefficients up to a given order for an arbitrary multidimensional driving process $X$ and its tree-iterated integrals can be very hard,
cf.  \cite[Section 4]{BB00} for a $2$-dimensional example. However, we already pointed out that in practice, one is not able to simulate the tree-iterated integrals anyway. 
We therefore make the same ansatz as in \cite{DNT12} and replace the tree-iterated integrals by products of increments. This simplifies the task of matching the coefficients a lot, and one is able to deduce
the order conditions, in principle, up to any order, cf. Theorem \ref{Thm_local_error} and the subsequent remark. We call such schemes \emph{simplified Runge-Kutta methods}.
As in \cite{DNT12}, the Wong-Zakai error plays a fundamental role in their convergence analysis. Loosely speaking, our main result is the following:

\begin{theorem*}
 Let $\mathbf{X}$ be an $\alpha$-H\"older rough path (branched or geometric) and assume that $f_0$ and $f$ are sufficiently smooth and bounded with bounded derivatives.
 If the Wong-Zakai error to approximate \eqref{eqn:RDE_with_drift} is of order $r_0$, a simplified Runge-Kutta method \eqref{eqn:RK_equations_simple} of order $p$ converges with rate \mbox{$\min\{(p+1)\alpha - 1, r_0\}$}. 
\end{theorem*}

As an application, we can study the scheme when the driving process is a fractional Brownian motion with Hurst parameter $H \in (1/4,1/2]$. In this case, the Wong-Zakai error is arbitrarily close
to $2H -1/2$, cf. \cite{FR14}. We therefore obtain:

\begin{corollary*}
 For a fractional Brownian motion with Hurst parameter $H \in (1/4,1/2]$, a simplified Runge-Kutta scheme of order $3$ converges with rate arbitrarily close to $2H - 1/2$.
\end{corollary*}

We already pointed out that numerical schemes studied in the context of rough paths theory are mostly of Taylor-type. To our knowledge, the only exception is the article by
Hong, Huang and Wang \cite{HHW18} where a class of symplectic Runge-Kutta methods is considered to solve Hamiltonian equations driven by Gaussian processes.
Our article differs from  \cite{HHW18} in several regards. On the technical level, no $B$-series are used in  \cite{HHW18}, the authors have to prove all necessary
estimates ``by hand'' in the framework of geometric rough paths. Consequently, they do not provide general order conditions. For instance, no explicit Runge-Kutta methods are deduced in \cite{HHW18}.
Moreover, their approach is probably hard to generalize to schemes of arbitrary order, whereas our approach does not have any limitations in this regard. \smallskip

The article is structured as follows. In Section \ref{sec:full_RK}, we define the equations which can be expanded to obtain the desired $B$-series.
Section \ref{sec:BRP} explains the concept of branched rough paths, deduces the $B$-series representation for equation \eqref{eqn:RDE_with_drift} and discusses the local error of full Runge-Kutta methods.
Simplified Runge-Kutta methods are defined in Section \ref{sec:simple_RK}, where the necessary order conditions are derived to obtain the local error of the numerical scheme.
In Section \ref{sec:global_rates}, we deduce the global error for our methods. The article closes with numerical experiments presented in Section \ref{sec:numerics}. \smallskip

Let us finally mention that in the whole article, we will discard the drift in \eqref{eqn:RDE_with_drift} and consider equations of the form
\begin{align}\label{eqn:RDE_without_drift}
 dy(t) = f(y(t))\, d\mathbf{X}(t),\quad y(t_0) = y_0,
\end{align}
only which simplifies the exposition a lot. Furthermore, this is not a real limitation if we assume that the first component of $X$ is just the path $t \mapsto t$.

\section*{Notation and basic definitions}

\subsection*{General notation} Let $I$ be an interval in $\R$ and $V$ be a linear space. We call a function $X \colon I \to V$ a \emph{path} and $X_{s,t} := X(t) - X(s)$ an \emph{increment}. For a general two-parameter function $X \colon I \times I \to V$, we will often write $X_{s,t}$ instead of $X(s,t)$. If $|\cdot|$ is a norm on $V$ and $X \colon I \times I \to V$, we set
\begin{align*}
 \|X\|_{\alpha} := \|X\|_{\alpha;I} := \sup_{\stackrel{s,t \in I}{s \neq t}} \frac{|X_{s,t}|}{|t-s|^{\alpha}} 
\end{align*}
for $\alpha \in (0,1]$ and call it the \emph{$\alpha$-H\"older (semi-)norm} of $X$. For $x \in \R$, we use the notation $\lfloor x \rfloor := \max \{k \in \Z\,:\, k \leq x \}$.  
Let $\gamma > 0$ and $\gamma = [\gamma] + \{\gamma\}$ where $[\gamma]$ is an integer and $\{ \gamma \} \in (0,1]$. We will say that a vector field $f \colon \R^m \to \R^m$ belongs to the class $\operatorname{Lip}^{\gamma}$ if $f$ is $[\gamma]$-times continuously differentiable and the $[\gamma]$-th derivative is locally $\{\gamma\}$-H\"older continuous. $f$ is of class $\operatorname{Lip}^{\gamma}_b$ if, in addition, $f$ and all its derivatives are bounded and if the $[\gamma]$-th derivative is globally $\{\gamma\}$-H\"older continuous. More generally, a collection of vector fields $f = (f_1,\ldots,f_d)$ is of class $\operatorname{Lip}^{\gamma}$ resp. $\operatorname{Lip}^{\gamma}_b$ if every $f_i$, $i = 1, \ldots, d$, is of class $\operatorname{Lip}^{\gamma}$ resp. $\operatorname{Lip}^{\gamma}_b$.

\subsection*{Trees and the Connes-Kreimer Hopf algebra} Let $\mathcal{T}$ be the set of all rooted, labeled trees with vertex decorations from the set $\{1,\ldots,d\}$. We will use a recursive definition to construct trees. The empty tree will be denoted by $1$. We use the convention that $1 \notin \mathcal{T}$ and set $\mathcal{T}^0 := \mathcal{T} \cup \{1\}$. If $\tau_1,\ldots \tau_m \in \mathcal{T}^0$, $[\tau_1 \cdots \tau_m]_a$ denotes the tree obtained by attaching all trees $\tau_1,\ldots \tau_m$ to a new vertex which we label by $a  \in \{1,\ldots,d\}$. We use the notation $\bullet _a = [1]_a$ for the single vertex tree with label $a$. The order of the branches of the tree does not matter, i.e. $[\tau_{\sigma(1)} \cdots \tau_{\sigma(m)}]_a = [\tau_{1} \cdots \tau_{m}]_a$ holds for every permutation $\sigma$. If $\tau \in \mathcal{T}^0$ is a tree, $|\tau|$ denotes the number of vertices. The set $\mathcal{T}_N$ consists of all trees $\tau \in \mathcal{T}$ such that $|\tau| \leq N$ and we set $\mathcal{T}_N^0 := \mathcal{T}_N \cup \{1\}$. We let
\begin{align*}
 \mathcal{F} := \{ \tau_1 \cdots \tau_m\, :\, \tau_i \in \mathcal{T},\ m \in \N \}
\end{align*}
denote the set of unordered forests and define $\mathcal{F}^0 := \mathcal{F} \cup \{1\}$. The map $| \cdot |$ is extended to $\mathcal{F}^0$ by setting
\begin{align*}
 |\tau_1 \cdots \tau_m| := |\tau_1| + \ldots + |\tau_m|.
\end{align*}
As before, $\mathcal{F}_N$ contains all $\mathfrak{h} \in \mathcal{F}$ with $|\mathfrak{h}| \leq N$ and $\mathcal{F}_N^0 := \mathcal{F}_N \cup \{1\}$.

We define $(\mathcal{H},\cdot)$ to be the commutative polynomial algebra generated by the variables $\mathcal{T}$. Alternatively, we can view $\mathcal{H}$ as the real vector space spanned by the elements in $\mathcal{F}^0$. A coproduct $\Delta \colon \mathcal{H} \to \mathcal{H} {\otimes} \mathcal{H}$ is defined recursively by setting $\Delta 1 := 1 {\otimes} 1$ and 
\begin{align*}
 \Delta [\tau_1 \cdots \tau_m]_a := [\tau_1 \cdots \tau_m]_a {\otimes} 1 + (\operatorname{id} {\otimes} B_+^a) (\Delta \tau_1 \cdots \Delta \tau_m)
\end{align*}
for a tree $[\tau_1 \cdots \tau_m]_a$ where $B_+^a$ is the operator defined by $B_+^a(\tau_1 \cdots \tau_n) := [\tau_1 \cdots \tau_n]_a$ on the forest  $\tau_1 \cdots \tau_n$. We then extend the definition to forests by setting $\Delta (\tau_1 \cdots \tau_m) := \Delta \tau_1 \cdots \Delta \tau_m$ and eventually define $\Delta$ on $\mathcal{H}$ by linear extension. We will use Sweedler's notation 
\begin{align*}
 \Delta \mathfrak{h} = \sum_{(\mathfrak{h})} \mathfrak{h}^{(1)} \otimes \mathfrak{h}^{(2)}.
\end{align*}
One can also construct an antipode $S \colon \mathcal{H} \to \mathcal{H}$, i.e. a map which satisfies
\begin{align*}
 M(\operatorname{id} {\otimes} S) \Delta x =  M(S  {\otimes} \operatorname{id}) \Delta x = x
\end{align*}
for every $x \in \mathcal{H}$ where $M(x {\otimes} y) := xy$. Then, $(\mathcal{H},\cdot,\Delta,S)$ is called the \emph{Connes-Kreimer Hopf algebra} \cite{CK98}, cf. also \cite[Chapter 2]{HK15}. The dual Hopf algebra will be denoted by $(\mathcal{H}^*,\star,\delta,S^*)$. 

For a general account on Hopf algebras, cf. \cite{Swe69} or \cite{Abe80}. The Connes-Kreimer Hopf algebra is also discussed in \cite{Man06}.

%
%

\section{The full Runge-Kutta method}\label{sec:full_RK}

In this section, we will define $s$-stage Runge-Kutta methods. We follow the approach developed by Burrage and Burrage in \cite{BB00}. Let $Z^{(1)}, \ldots, Z^{(d)}$ be given $s\times s$-matrices and $z^{(1)},\ldots, z^{(d)}$ vectors in $\R^s$.
For given $y_n \in \R^m$, consider the equations
\begin{align}\label{eqn:RK_equations}
 \begin{split}
      Y_i &= y_n + \sum_{k = 1}^d \sum_{j=1}^s Z^{(k)}_{ij} f_k(Y_j) \\
      y_{n+1} &= y_n + \sum_{k = 1}^d \sum_{i=1}^s z^{(k)}_i f_k(Y_i).
 \end{split}
\end{align}
In applications, $Z$ and $z$ can (and will) be random. Moreover, both values will depend on the step size $h>0$ of the numerical scheme \eqref{eqn:RK_equations}. Note that the first equation can be implicit in which case the existence of a solution is not guaranteed. In fact, this question will depend on the properties of the vector fields $f_i$. For instance, it can be shown that solutions exist in case that all vector fields are bounded, cf. \cite[Proposition 4.1]{HHW18}. However, we will not address this question here and just assume that solutions exist.

Set 
{$\Phi(1)(h):= \mathbf 1_s := (1,\cdots,1)^T \in \R^s$} for $i = 1,\ldots,d$ and for a tree $\tau = [\tau_1 \cdots \tau_n]_i$,
\begin{align*}
 \Phi(\tau)(h) := \Pi_{j=1}^n (Z^{(i)} \Phi(\tau_j)(h)), \quad 
 {a(\tau)(h) := \langle z^{(i)}, \Pi_{k=1}^n\Phi(\tau_k)(h) \rangle.}
\end{align*}
Notice that above, the product of two vectors has to be understood component-wise. 

\begin{definition}\label{defn:elem_diff}
 Let $f = (f_1,\ldots,f_d)$ be sufficiently smooth such that all derivatives below exist. For a tree $\tau \in \mathcal{T}^0$, we define the \emph{elementary differentials} $F(\tau) \colon \R^m \to \R^m$ recursively by setting 
 \begin{itemize}
  \item[(i)] $F(1)(y) := y$,
  \item[(ii)] $F(\bullet_i)(y) := f_i(y)$\quad and
  \item[(iii)] $F(\tau)(y) := f_i^{(n)}(y) (F(\tau_1)(y),\ldots,F(\tau_n)(y))$ for a tree $\tau = [\tau_1 \cdots \tau_n]_i$ where $f_i^{(n)}$ denotes the $n$-th total derivative of $f_i$
 \end{itemize}
 for $y \in \R^m$.

\end{definition}

Next,we define some combinatoric quantities. For unlabeled trees, we set
\begin{align*}
 \gamma(1) = 0,\quad \gamma(\bullet) = 1,\quad \gamma([\tau_1,\ldots,\tau_k]) = |[\tau_1,\ldots,\tau_k]| \prod_{i=1}^k \gamma(\tau_i),
\end{align*}
and
\begin{align*}
 \beta(1) = 1, \quad \beta(\bullet) = 1, \quad \beta(\tau ) := \binom{|\tau| - 1}{|\tau_1|,\ldots, |\tau_k|} \frac{1}{r_1! \cdots r_q!} \prod_{j = 1}^k \beta(\tau_j)
\end{align*}
where $\tau = [\tau_1,\ldots,\tau_k] = [(\tau_1)^{r_1},\ldots,(\tau_q)^{r_q}]$, $\tau_1,\ldots,\tau_q$ being pairwise distinct trees. For labeled trees, we use the same definition. The main result in \cite{BB00} we are going to use is the following:
\begin{theorem}\label{B-series_RK}
 The Taylor series expansion of \eqref{eqn:RK_equations} is
 \begin{align}\label{eqn:taylor_BB}
  y_1 = y_0 + \sum_{\tau \in \mathcal{T}} \frac{\gamma(\tau)}{|\tau|!} \beta(\tau)  a(\tau)(h) F(\tau)(y_0).
 \end{align}
\end{theorem}
\begin{proof}
 \cite[Theorem 2.5]{BB00}.
\end{proof}
The coefficients in \eqref{B-series_RK} are sometimes noted in a different form which we recall now. Following \cite{Gub10}, we define the \emph{symmetry factor} $\sigma$ for unlabeled trees as
\begin{align*}
 \sigma(1) = 1, \quad \sigma(\bullet) = 1, \quad \sigma(\tau) = r_1! \cdots r_q! \prod_{j = 1}^k \sigma(\tau_j)
\end{align*}
where $\tau = [\tau_1,\ldots,\tau_k] = [(\tau_1)^{r_1},\ldots,(\tau_q)^{r_q}]$ with $\tau_1,\ldots,\tau_q$ being pairwise distinct trees. For labeled trees, the same definition is used.

\begin{lemma}\label{lemma_symmetry_factor}
  For every tree $\tau \in \mathcal{T}^0$,
 \begin{align*}
  \frac{1}{\sigma(\tau)} = \frac{\gamma(\tau)}{|\tau|!} \beta(\tau).
 \end{align*}

\end{lemma}

\begin{proof}
We prove this lemma by induction on the height of $\tau$ for unlabeled trees. The equality is true for $\tau = 1$ and $\tau = \bullet$. Let us assume that the claim is true for each sub-tree of 
$\tau = [\tau_1,\ldots,\tau_k] = [(\tau_1)^{r_1},\ldots,(\tau_q)^{r_q}]$. Then, we have
\begin{align*}
 \beta(\tau) &= \frac{(|\tau| - 1)!}{|\tau_1|! \cdots |\tau_k|!} \frac{1}{r_1! \cdots r_q!} \prod_{j = 1}^k \beta(\tau_j) \\
 &= \frac{(|\tau| - 1)!}{|\tau_1|! \cdots |\tau_k|!} \frac{1}{r_1! \cdots r_q!} \prod_{j = 1}^k \frac{|\tau_j|!}{\gamma(\tau_j)} \frac{1}{\sigma(\tau_j)} \\
 &= \frac{(|\tau| - 1)!}{r_1! \cdots r_q!} \frac{|\tau|}{\gamma(\tau)} \prod_{j = 1}^k \frac{1}{\sigma(\tau_j)} \\
 &= \frac{|\tau|!}{\gamma(\tau)} \frac{1}{\sigma(\tau)}.
\end{align*}

\end{proof}

\section{Branched rough paths and B-series expansion for rough differential equations}\label{sec:BRP}

In this section, we recall the concept of a \emph{branched rough paths} introduced by Gubinelli in \cite{Gub10}. We use a similar approach and notation as Hairer and Kelly in \cite{HK15}. Our main goal is to deduce the $B$-series expansion of \eqref{eqn:RDE_without_drift} which we will eventually achieve in Theorem \ref{B-series_RDE}. Note that Hairer and Kelly state a similar result in \cite[Proposition 3.8]{HK15}. However, we can not use their result for two reasons: First, it is not quantitative, i.e. the order of the truncation error is not specified. Second, it is not entirely correct since the expansion in \cite[Proposition 3.8]{HK15} lacks the symmetry factor. The proof of \cite[Proposition 3.8]{HK15} was corrected in the Master thesis of Rosa Prei{\ss}, and we are grateful to her for providing us with the corrected version.

\begin{definition}\label{defn_branched_rp}
 Let $\alpha \in (0,1]$. A \emph{$\alpha$-branched rough path} is a map $\mathbf{X} \colon [0,T] \times [0,T] \to \mathcal{H}^*$ such that 
 \begin{enumerate}
  \item for all $s, t \in [0,T]$ and all $\mathfrak{h}_1,\mathfrak{h}_2 \in \mathcal{H}$,
  \begin{align*}
   \langle \mathbf{X}_{s,t}, \mathfrak{h}_1 \rangle \langle \mathbf{X}_{s,t},\mathfrak{h}_2 \rangle = \langle \mathbf{X}_{s,t}, \mathfrak{h}_1 \cdot \mathfrak{h}_2 \rangle
  \end{align*}
  
  \item for all $s,u,t \in [0,T]$,
  \begin{align*}
   \mathbf{X}_{s,t} = \mathbf{X}_{s,u} \star \mathbf{X}_{u,t}
  \end{align*}
  
  \item for every $\tau \in \mathcal{T}$,
  \begin{align*}
   \sup_{s \neq t} \frac{ | \langle \mathbf{X}_{s,t}, \tau \rangle |}{|t-s|^{\alpha |\tau|}} < \infty.
  \end{align*}

 \end{enumerate}
 
 The space of $\alpha$-branched rough paths will be denoted by $\mathscr{C}^{\alpha}([0,T],\R^d)$. It is a complete metric space with metric
 \begin{align*}
  \varrho_{\alpha}(\mathbf{X},\mathbf{Y}) := \sum_{\tau \in \mathcal{T}_N} \sup_{s \neq t} \frac{ | \langle \mathbf{X}_{s,t} - \mathbf{Y}_{s,t},\tau \rangle |}{|t-s|^{\alpha |\tau|}}
 \end{align*}
 where $N = \lfloor 1/\alpha \rfloor$.

\end{definition}

\begin{example}\label{example_smooth}
 Let $X = (X^1 ,\ldots, X^d) \colon [0,T] \to \R^d$ be a piece-wise $C^1$-path. We define $\mathbf{X}$ by setting
 \begin{align*}
  \langle \mathbf{X}_{s,t}, \bullet_i \rangle = X^i(t) - X^i(s) \quad \text{and} \quad  \langle \mathbf{X}_{s,t}, [\tau_1 \cdots \tau_n]_i \rangle = \int_s^t \langle \mathbf{X}_{s,u}, \tau_1 \rangle \cdots \langle \mathbf{X}_{s,u}, \tau_n \rangle\, dX^i(u)
 \end{align*}
 for any tree $\tau = [\tau_1 \cdots \tau_n]_i \in \mathcal{T}$ and
 \begin{align*}
  \langle \mathbf{X}_{s,t}, \tau_1 \cdots \tau_n \rangle = \langle \mathbf{X}_{s,t}, \tau_1 \rangle \cdots \langle \mathbf{X}_{s,t}, \tau_n \rangle
 \end{align*}
 for any forest $\tau_1 \cdots \tau_n$. We can now extend $\mathbf{X}$ linearly to a map on $\mathcal{H}$ and therefore obtain a map $\mathbf{X} \colon [0,T] \times [0,T] \to \mathcal{H}^*$. It can be shown that this map defines a $\alpha$-branched rough path for every $\alpha \in (0,1]$. If $X$ is $\alpha$-H\"older continuous for some $\alpha \in (1/2,1]$, we can use the Young integral to define $\mathbf{X}$ as above. In this case, it can be shown that $\mathbf{X}$ defines a $\alpha'$-branched rough path for every $\alpha' \in (0,\alpha]$. In this case, $\mathbf{X}$ is called the \emph{natural lift} of the (smooth) path $X$ to a branched rough path.

\end{example}

Next, we define a class of paths which we can integrate against a branched rough path.

\begin{definition}
 Let $\mathbf{X}$ be a $\alpha$-branched rough path and $N = \lfloor 1/\alpha \rfloor$. A path $\mathbf{Z} \colon [0,T] \to \mathcal{H}_{N-1}$ satisfying
 \begin{align}\label{eqn:controlled}
  \langle \mathfrak{h}, \mathbf{Z}(t) \rangle = \langle \mathbf{X}_{s,t} \star \mathfrak{h}, \mathbf{Z}(s) \rangle + R^\mathfrak{h}_{s,t}
 \end{align}
 for each $\mathfrak{h} \in \mathcal{F}^0_{N-1}$ where $|R^\mathfrak{h}_{s,t}| \leq C |t-s|^{(N - |\mathfrak{h}|)\alpha}$ is called \emph{controlled by $\mathbf{X}$}. We will also say that $\mathbf{Z}$ is a \emph{controlled path above the path $t \mapsto Z(t) := \langle 1, \mathbf{Z}(t) \rangle$}. More generally, we call a path $\mathbf{Z} \colon [0,T] \to (\mathcal{H}_{N-1})^m$ \emph{controlled by $\mathbf{X}$} if \eqref{eqn:controlled} holds, understood as an equation in $\R^m$. The space of controlled paths $\mathbf{Z} \colon [0,T] \to (\mathcal{H}_{N-1})^m$ will be denoted by $\mathcal{Q}_{\mathbf{X}}(\R^m)$ which is a Banach space with the norm
 \begin{align*}
  \| \mathbf{Z} \|_{\mathcal{Q}_{\mathbf{X}}(\R^m)} := |\mathbf{Z}(0)| + \sum_{\mathfrak{h} \in \mathcal{F}^0_{N-1}} \|R^\mathfrak{h}\|_{(N - |\mathfrak{h}|)\alpha}.
 \end{align*}

\end{definition}

The following lemma is given as an exercise in \cite{HK15}. We provide a full proof here for the reader's convenience.

\begin{lemma}\label{lemma:check_sewing}
 Let $\mathbf{X}$ be a $\alpha$-branched rough path and $\mathbf{Z}$ be controlled by $\mathbf{X}$. Set 
 \begin{align*}
  \tilde{Z}_{s,t} := \sum_{\mathfrak{h} \in \mathcal{F}_{N-1}^0} \langle \mathfrak{h}, \mathbf{Z}(s) \rangle \langle \mathbf{X}_{s,t},[\mathfrak{h}]_{i} \rangle.
 \end{align*}
 Then,
 \begin{align*}
  \tilde{Z}_{s,t} - \tilde{Z}_{s,u} - \tilde{Z}_{u,t} = - \sum_{\mathfrak{h} \in \mathcal{F}^0_{N-1}} \langle \mathbf{X}_{u,t}, [\mathfrak{h}]_i \rangle R^{\mathfrak{h}}_{s,u}.
 \end{align*}

\end{lemma}

\begin{proof}
 Let $\tilde{\mathfrak{h}} \in \mathcal{F}^0_{N-1}$. Then,
 \begin{align*}
  \langle \mathbf{X}_{s,t} \star \tilde{\mathfrak{h}}, \mathbf{Z}(s) \rangle &= \sum_{\mathfrak{h} \in \mathcal{F}_{N-1}^0} \langle \mathbf{X}_{s,t} \star \tilde{\mathfrak{h}}, \mathfrak{h} \rangle \langle \mathfrak{h}, \mathbf{Z}(s) \rangle \\
  &= \sum_{\mathfrak{h} \in \mathcal{F}_{N-1}^0} \langle \mathbf{X}_{s,t} \otimes \tilde{\mathfrak{h}}, \Delta \mathfrak{h} \rangle \langle \mathfrak{h}, \mathbf{Z}(s) \rangle \\
  &= \sum_{\mathfrak{h} \in \mathcal{F}_{N-1}^0} \sum_{(\mathfrak{h})} \langle \mathbf{X}_{s,t} \otimes \tilde{\mathfrak{h}}, \mathfrak{h}^{(1)} {\otimes} \mathfrak{h}^{(2)} \rangle \langle \mathfrak{h}, \mathbf{Z}(s) \rangle \\
  &= \sum_{\mathfrak{h} \in \mathcal{F}_{N-1}^0} \sum_{(\mathfrak{h})} \mathbf{1}_{\{\mathfrak{h}^{(2)} = \tilde{\mathfrak{h}}\}} \langle \mathbf{X}_{s,t}, \mathfrak{h}^{(1)} \rangle \langle \mathfrak{h}, \mathbf{Z}(s) \rangle.
 \end{align*}
    For $\mathfrak{h} \in \mathcal{F}^0_{N-1}$,
    \begin{align*}
     \langle \mathbf{X}_{s,t}, [\mathfrak{h}]_i \rangle &= \langle \mathbf{X}_{s,u} \star \mathbf{X}_{u,t}, [\mathfrak{h}]_i \rangle \\
     &= \langle \mathbf{X}_{s,u} \otimes \mathbf{X}_{u,t}, \Delta [\mathfrak{h}]_i \rangle \\
     &= \langle \mathbf{X}_{s,u} \otimes \mathbf{X}_{u,t}, [\mathfrak{h}]_i \otimes 1 \rangle + \langle \mathbf{X}_{s,u} \otimes \mathbf{X}_{u,t}, (\operatorname{id} \otimes B_+^i) \Delta \mathfrak{h} \rangle \\
     &= \langle \mathbf{X}_{s,u} , [\mathfrak{h}]_i  \rangle + \sum_{(\mathfrak{h})} \langle \mathbf{X}_{s,u}, \mathfrak{h}^{(1)} \rangle \langle \mathbf{X}_{u,t}, [\mathfrak{h}^{(2)}]_i \rangle.
    \end{align*}
    
    It follows that
    \begin{align*}
     \tilde{Z}_{s,t} - \tilde{Z}_{s,u} - \tilde{Z}_{u,t} &= \sum_{\mathfrak{h} \in \mathcal{F}^0_{N-1}} \langle \mathfrak{h}, \mathbf{Z}(s) \rangle \left( \langle \mathbf{X}_{s,t}, [\mathfrak{h}]_i \rangle - \langle \mathbf{X}_{s,u}, [\mathfrak{h}]_i \rangle \right) - \langle \mathfrak{h}, \mathbf{Z}(u) \rangle \langle \mathbf{X}_{u,t},[\mathfrak{h}]_{i} \rangle \\
     &= \sum_{\mathfrak{h} \in \mathcal{F}^0_{N-1}} \sum_{(\mathfrak{h})} \langle \mathfrak{h}, \mathbf{Z}(s) \rangle \langle \mathbf{X}_{s,u}, \mathfrak{h}^{(1)} \rangle \langle \mathbf{X}_{u,t}, [\mathfrak{h}^{(2)}]_i \rangle - \langle \mathfrak{h}, \mathbf{Z}(u) \rangle \langle \mathbf{X}_{u,t},[\mathfrak{h}]_{i} \rangle \\
     &= \sum_{\mathfrak{h} \in \mathcal{F}^0_{N-1}} \sum_{(\mathfrak{h})} \sum_{\tilde{\mathfrak{h}} \in \mathcal{F}^0_{N-1}} \mathbf{1}_{\{\mathfrak{h}^{(2)} = \tilde{\mathfrak{h}}\}} \langle \mathfrak{h}, \mathbf{Z}(s) \rangle \langle \mathbf{X}_{s,u}, \mathfrak{h}^{(1)} \rangle \langle \mathbf{X}_{u,t}, [\tilde{\mathfrak{h}}]_i \rangle - \langle \mathfrak{h}, \mathbf{Z}(u) \rangle \langle \mathbf{X}_{u,t},[\mathfrak{h}]_{i} \rangle \\
     &= \sum_{\tilde{\mathfrak{h}} \in \mathcal{F}^0_{N-1}} \sum_{\mathfrak{h} \in \mathcal{F}^0_{N-1}}\sum_{(\mathfrak{h})}  \mathbf{1}_{\{\mathfrak{h}^{(2)} = \tilde{\mathfrak{h}}\}} \langle \mathfrak{h}, \mathbf{Z}(s) \rangle \langle \mathbf{X}_{s,u}, \mathfrak{h}^{(1)} \rangle \langle \mathbf{X}_{u,t}, [\tilde{\mathfrak{h}}]_i \rangle - \langle \tilde{\mathfrak{h}}, \mathbf{Z}(u) \rangle \langle \mathbf{X}_{u,t},[\tilde{\mathfrak{h}}]_{i} \rangle \\
     &= \sum_{\tilde{\mathfrak{h}} \in \mathcal{F}^0_{N-1}} \langle \mathbf{X}_{u,t}, [\tilde{\mathfrak{h}}]_i \rangle \left( \langle \mathbf{X}_{s,u} \star \tilde{\mathfrak{h}}, \mathbf{Z}(s) \rangle - \langle \tilde{\mathfrak{h}}, \mathbf{Z}(u) \rangle \right) \\
     &= - \sum_{\tilde{\mathfrak{h}} \in \mathcal{F}^0_{N-1}} \langle \mathbf{X}_{u,t}, [\tilde{\mathfrak{h}}]_i \rangle R^{\tilde{\mathfrak{h}}}_{s,u}.
    \end{align*}

\end{proof}

\begin{theorem}[Gubinelli]
 Let $T > 0$, $\mathbf{X}$ be a $\alpha$-branched rough path and $\mathbf{Z}$ be controlled by $\mathbf{X}$. Then,
 \begin{align*}
  \int_s^t Z(r)\, d\mathbf{X}^i(r) := \lim_{|\mathcal{P}| \to 0} \sum_{[u,v] \in \mathcal{P}} \tilde{Z}_{u,v}
 \end{align*}
 exists for every $i \in \{1,\ldots, d\}$ and $[s,t] \subseteq [0,T]$ where
 \begin{align*}
  \tilde{Z}_{u,v} := \sum_{\mathfrak{h} \in \mathcal{F}_{N-1}^0} \langle \mathfrak{h}, \mathbf{Z}(u) \rangle \langle \mathbf{X}_{u,v},[\mathfrak{h}]_{i} \rangle
 \end{align*}
 and $\mathcal{P}$ denotes a partition of $[s,t]$ with mesh size $|\mathcal{P}|$. Moreover,there exists a constant $C$ depending only on $\alpha$ and $T$ such that
  \begin{align*}
   \left|\int_s^t Z(r)\, d\mathbf{X}^i(r) - \tilde{Z}_{s,t} \right| \leq C |t-s|^{(N+1)\alpha} \sum_{\mathfrak{h} \in \mathcal{F}^0_{N-1}} \| \langle \mathbf{X}_{\cdot,\cdot}, [\mathfrak{h}]_i \rangle \|_{(|\mathfrak{h}| + 1)\alpha;[s,t]} \| R^{\mathfrak{h}} \|_{(N - |\mathfrak{h}|)\alpha; [s,t]}.
  \end{align*}

\end{theorem}

\begin{proof}
 This is a consequence of the sewing lemma \cite[Lemma 4.2]{FH14} and Lemma \ref{lemma:check_sewing}.
\end{proof}

 The above theorem defines a map which sends a controlled path $\mathbf{Z}$ to a path $t \mapsto \int_0^t Z(r)\, d\mathbf{X}^i(r) \in \R^m$. In fact, this map can be naturally extended to a map $\mathbf{Z} \mapsto  \int_0^{\cdot} \mathbf{Z}(r)\, d\mathbf{X}^i(r)$ where $t \mapsto \int_0^{t} \mathbf{Z}(r)\, d\mathbf{X}^i(r)$ is a controlled path above $t \mapsto \int_0^{t} Z(r)\, d\mathbf{X}^i(r)$. To do this, we have to specify $\langle \mathfrak{h}, \int_0^{t} \mathbf{Z}(r)\, d\mathbf{X}^i(r) \rangle$ for every dual element $\mathfrak{h} \in \mathcal{F}^*_{N-1} \cup \{1\}$. We set 
 \begin{align*}
  \langle 1, \int_0^{t} \mathbf{Z}(r)\, d\mathbf{X}^i(r) \rangle := \int_0^t Z(r)\, d\mathbf{X}^i(r)
 \end{align*}
 and 
 \begin{align*}
   \langle [\tau_1 \cdots \tau_n]_i, \int_0^{t} \mathbf{Z}(r)\, d\mathbf{X}^i(r) \rangle := \langle \tau_1 \cdots \tau_n, \mathbf{Z}(t) \rangle.
 \end{align*}
 In all other cases, we define
 \begin{align*}
  \langle \tau_1 \cdots \tau_n, \int_0^{t} \mathbf{Z}(r)\, d\mathbf{X}^i(r) \rangle := 0.
 \end{align*}
 More generally, if $\mathbf{Z} = (\mathbf{Z}^1,\ldots,\mathbf{Z}^d)$ and every $\mathbf{Z}^i$ is controlled by $\mathbf{X}$, we define a controlled path $t \mapsto \int_0^t \mathbf{Z}(r) \cdot d \mathbf{X}(r)$ by setting 
 \begin{align*}
  \langle 1, \int_0^t \mathbf{Z}(r) \cdot d \mathbf{X}(r) \rangle := \sum_{i = 1}^d \langle 1, \int_0^t \mathbf{Z}^i(r)\, d\mathbf{X}^i(r) \rangle,
 \end{align*}
 \begin{align*}
  \langle [\tau_1 \cdots \tau_n]_i, \int_0^t \mathbf{Z}(r) \cdot d \mathbf{X}(r) \rangle := \langle \tau_1 \cdots \tau_n, \mathbf{Z}^i(t) \rangle
 \end{align*}
 for a tree $[\tau_1 \cdots \tau_n]_i \in \mathcal{T}_{N-1}$, $i \in \{1,\ldots,d\}$ and 
\begin{align*}
  \langle \tau_1 \cdots \tau_n, \int_0^{t} \mathbf{Z}(r) \cdot d\mathbf{X}(r) \rangle := 0
 \end{align*}
 otherwise.
 
 \begin{theorem}
  The map
  \begin{align*}
   I \colon \mathcal{Q}_{\mathbf{X}}(\R^m)^d &\to \mathcal{Q}_{\mathbf{X}}(\R^m) \\
   \mathbf{Z} &\mapsto \int_0^{\cdot}  \mathbf{Z}(r) \cdot d\mathbf{X}(r)
  \end{align*}
  is well-defined and continuous.

 \end{theorem}
 
 \begin{proof}
  \cite[Theorem 8.5]{Gub10}.
 \end{proof}
 
 The next lemma shows that controlled paths composed with sufficiently smooth functions are again controlled.
 
 \begin{lemma}
  Let $\phi \colon \R^m \to \R^m$ be sufficiently smooth such that all derivatives below exist. For $\mathbf{Z} \in \mathcal{Q}_{\mathbf{X}}(\R^m)$, we define $\langle 1,   \phi(\mathbf{Z}(t)) \rangle := \phi(Z(t))$ and
  \begin{align*}
   \langle \mathfrak{h}, \phi(\mathbf{Z}(t)) \rangle := \sum_{n = 1}^{N-1} \sum_{\mathfrak{h}_1 \cdots \mathfrak{h}_n = \mathfrak{h}} \frac{1}{n!} \phi^{(n)}(Z(t)) \left( \langle \mathfrak{h}_1, \mathbf{Z}(t) \rangle, \ldots, \langle \mathfrak{h}_n, \mathbf{Z}(t) \rangle \right)
  \end{align*}
  for $\mathfrak{h} \in \mathcal{F}_{N-1}^*$. Then, $\phi(\mathbf{Z}) \in \mathcal{Q}_{\mathbf{X}}(\R^m)$. 

 \end{lemma}
 
 \begin{proof}
  \cite[Lemma 8.4]{Gub10}.
 \end{proof}
 
 We are now able to say what a solution to \eqref{eqn:RDE_without_drift} actually means.
 
 \begin{definition}
  A path $y \colon [0,T] \to \R^m$ is a \emph{solution to \eqref{eqn:RDE_without_drift}} if $y(t_0) = y_0$ and if there exists a controlled path $\mathbf{Y} \in \mathcal{Q}_{\mathbf{X}}(\R^m)$ above $y$ such that
  \begin{align}\label{eqn:branched_RDE}
   \mathbf{Y}(t) - \mathbf{Y}(s) = \int_s^t f(\mathbf{Y}(r)) \cdot d \mathbf{X}(r)
  \end{align}
  holds for every $s \leq t$, $s,t \in [t_0,T]$, where we set $f(\mathbf{Y}(r)) := (f_1(\mathbf{Y}(r)), \ldots f_d(\mathbf{Y}(r)))$.

 \end{definition}
 
 Proving that \eqref{eqn:branched_RDE} admits a (unique) solution is done by a standard fixed-point argument \cite[Theorem 8.8]{Gub10}. If $f$ is of class $\operatorname{Lip}^{\gamma-1}$ for some $\gamma > \frac{1}{\alpha}$, a local solution to \eqref{eqn:branched_RDE} exists. If $\operatorname{Lip}^{\gamma-1}_b$, the solution exists on every time interval. For $f$ being of class $\operatorname{Lip}^{\gamma}$ resp. $\operatorname{Lip}^{\gamma}_b$, the local resp. global solution is unique. Moreover, in the second case, the solution map is continuous.
 
 Recall the definition of the elementary differential $F(\tau)$ for $\tau \in \mathcal{T}^0$ given in Definition \ref{defn:elem_diff}. 
 

\begin{lemma}\label{lemma:coeff_sol}
 Let $\mathbf{Y} \colon [0,T] \to \mathcal{H}_{N-1}$ with $y(t) = \langle 1, \mathbf{Y}(t) \rangle$ be a solution to \eqref{eqn:branched_RDE}. Then, the coefficients of $\mathbf{Y}$ are given by 
  \begin{align*}
   \langle \tau, \mathbf{Y}(t) \rangle = \frac{1}{\sigma(\tau)} F(\tau)(y(t))
  \end{align*}
  for $\tau \in \mathcal{T}^*_{N-1} \cup \{1\}$ and $\langle \tau_1 \cdots \tau_n, \mathbf{Y}(t) \rangle = 0$ for $\tau_1 \cdots \tau_n \in \mathcal{F}^*_{N-1} \setminus \mathcal{T}^*_{N-1}$.
\end{lemma}

\begin{proof}
 Being a solution to \eqref{eqn:branched_RDE} means that
 \begin{align*}
  y(t) - y(s) = \langle 1, \int_s^t f(\mathbf{Y}(r)) \cdot d\mathbf{X}(r) \rangle 
 \end{align*}
 and
 \begin{align}\label{eqn:rec_sol}
  \langle [\tau_1 \cdots \tau_n]_i , \mathbf{Y}(t) \rangle = \langle [\tau_1 \cdots \tau_n]_i , \int_0^t f(\mathbf{Y}(r)) \cdot d\mathbf{X}(r) \rangle 
 \end{align}
 for all $[\tau_1 \cdots \tau_n]_i \in \mathcal{T}^*_{N-1} $, $i \in \{1,\ldots,d\}$, and
 \begin{align*}
  \langle \tau_1 \cdots \tau_n , \mathbf{Y}(t) \rangle = 0 
 \end{align*}
 for all $\tau_1 \cdots \tau_n \in \mathcal{F}^*_{N-1} \setminus \mathcal{T}^*_{N-1}$. We prove the assertion for all trees $\tau \in \mathcal{T}^*_{N-1}$ by induction on the height of $\tau$. For $\tau = 1$, the claim follows by definition. Now let $\tau = [\tau_1 \cdots \tau_n]_i = [(\tau_1)^{r_1} \cdots (\tau_q)^{r_q}]_i \in \mathcal{T}^*_{N-1}$ for some $i \in \{1,\ldots,d\}$ and pairwise distinct trees $\tau_1,\ldots,\tau_q$. From \eqref{eqn:rec_sol}, we have
  \begin{align*}
   \langle \tau , \mathbf{Y}(t) \rangle &= \langle [\tau_1 \cdots \tau_n]_i , \int_0^t f(\mathbf{Y}(r)) \cdot d\mathbf{X}(r) \rangle \\
   &= \langle \tau_1 \cdots \tau_n, f_i(\mathbf{Y}(t)) \rangle \\
   &=  \sum_{\substack{ \lambda_1, \cdots, \lambda_n \in \mathcal{T}^*_{N-2} \\ \lambda_1 \cdots \lambda_n = \tau_1 \cdots \tau_n}} \frac{1}{n!} f_i^{(n)}(y(t)) \left( \langle \lambda_1, \mathbf{Y}(t) \rangle, \ldots,  \langle \lambda_n, \mathbf{Y}(t) \rangle \right) \\
   &= \frac{1}{r_1! \cdots r_q!} \frac{1}{n!} \sum_{\sigma \in \operatorname{sym}(n)} f_i^{(n)}(y(t)) \left( \langle \tau_{\sigma(1)}, \mathbf{Y}(t) \rangle, \ldots,  \langle \tau_{\sigma(n)}, \mathbf{Y}(t) \rangle \right) \\
   &= \frac{1}{r_1! \cdots r_q!} f_i^{(n)}(y(t)) \left( \langle \tau_1, \mathbf{Y}(t) \rangle, \ldots,  \langle \tau_n, \mathbf{Y}(t) \rangle \right) \\
   &= \frac{1}{r_1! \cdots r_q!} \frac{1}{\sigma(\tau_1) \cdots \sigma(\tau_n)} f_i^{(n)}(y(t)) \left( F(\tau_1)(y(t)), \ldots,  F(\tau_n)(y(t))\right) \\
   &= \frac{1}{\sigma(\tau)} F(\tau)(y(t))
  \end{align*}
  by induction hypothesis.

\end{proof}

 \begin{theorem}\label{B-series_RDE}
 Let $h>0$. 
 Then, \eqref{eqn:RDE_without_drift} has the expansion
  \begin{align*}
   y(t_0 + h) 
   =y_0 + \sum_{\tau \in \mathcal{T}_p} \frac{1}{\sigma(\tau)} F({\tau})(y_0) \langle \mathbf{X}_{t_0, t_0 + h},\tau \rangle + \mathcal{O}(h^{(p+1)\alpha})
  \end{align*}
  for every $p \geq \lfloor 1/\alpha \rfloor$.

 \end{theorem}
 
 \begin{proof}
  Let $s < t$. Note that
  \begin{align*}
   y(t) - y(s) = \sum_{i = 1}^d \int_s^t f_i(y(r))\, d\mathbf{X}^i(r).
  \end{align*}
  For $i \in \{1,\ldots, d\}$, set 
  \begin{align*}
   \tilde{Z}_{s,t}^i := \sum_{\mathfrak{h} \in \mathcal{F}_{p-1}^0} \langle \mathfrak{h}, f_i(\mathbf{Y}(s)) \rangle \langle \mathbf{X}_{s,t},[\mathfrak{h}]_i \rangle  = \sum_{\mathfrak{h} \in \mathcal{F}_{p-1}^0} \frac{1}{\sigma([\mathfrak{h}]_i)} F([\mathfrak{h}]_i)(y(s)) \langle \mathbf{X}_{s,t},[\mathfrak{h}]_i \rangle
  \end{align*}
  where we use Lemma \ref{lemma:coeff_sol} for the equality. We therefore obtain
  \begin{align*}
   y(t) - y(s) =  \sum_{i = 1}^d \tilde{Z}_{s,t}^i + R_{s,t} = \sum_{\tau \in \mathcal{T}_{p}} \frac{1}{\sigma(\tau)} F(\tau)(y(s)) \langle \mathbf{X}_{s,t},\tau \rangle + R_{s,t}
  \end{align*}
  where 
  \begin{align*}
   R_{s,t} = \sum_{i = 1}^d \int_s^t f_i(y(r))\, d\mathbf{X}^i(r) - \tilde{Z}_{s,t}^i.
  \end{align*}
  Using Lemma \ref{lemma:check_sewing} and the sewing Lemma \cite[Lemma 4.2]{FH14}, we conclude that $R_{s,t}$ is of order $\mathcal{O}((t-s)^{(p+1)\alpha})$.

 \end{proof}

We introduce the local error by $le(t_0, y_0; h) : = y(t_0 + h) - y_1$ which is the error of one step with the iterative scheme \eqref{eqn:RK_equations} starting in the exact value $y_0$.
Comparing Theorems \ref{B-series_RK} and \ref{B-series_RDE} and exploiting Lemma \ref{lemma_symmetry_factor}, we see that 
the local error is \begin{align}\label{rate_local}
        |le(t_0, y_0; h)| = \mathcal{O}(h^{(p+1)\alpha})                                   
                                          \end{align}
for sufficiently small $h>0$ if and only if 
\begin{align}\label{cond_loc_rate}
 \langle \mathbf{X}_{t_0, t_0+h},\tau \rangle  = a(\tau)(h)\quad \forall \tau \in \mathcal{T} \text{ with } |\tau|\leq p.
\end{align}

\section{Simplified Runge Kutta Methods}\label{sec:simple_RK}

In the following, $\mathbf X$ denotes an $\alpha$-branched rough path for some $\alpha \in (0,1]$. Assume that there is a smooth path $X^h$ such that its natural lift $\mathbf X^h$ to a branched rough path, cf. Example \ref{example_smooth}, approximates $\mathbf X$, i.e., $\varrho_\alpha(\mathbf X^h, \mathbf X)\rightarrow 0$ for $h\rightarrow 0$. This implies that $\mathbf{X}$ is a \emph{geometric} rough path \cite[Section 4]{HK15} and that $\varrho^g_\alpha(\mathbf X^h, \mathbf X)\rightarrow 0$ where $\varrho^g_{\alpha}$ denotes the inhomogeneous rough paths metric for geometric rough paths \cite{FV10}. We introduce the equation associated to the smooth driver by 
\begin{align}\label{eqn:ODE}
 dy^h(t) = f(y^h(t))\, d\mathbf{X}^h(t),\quad y^h(t_0) = y_0.
\end{align}
This equation can be solved by considering $\mathbf{X}$ as a branched rough path or a geometric rough path, the solution is the same in both cases. It also coincides with the solution to the corresponding Riemann-Stieltjes equation which is well-defined since $X^h$ is smooth by assumption. Since the solution to \eqref{eqn:RDE_without_drift} is a locally Lipschitz continuous function of $\mathbf X$, cf. \cite{FV10} in the case of geometric rough paths or \cite[Theorem 8.8]{Gub10} for branched rough paths, i.e., 
\begin{align}\label{eqn:local_hoelder_rp}
 \sup_{t\in [t_0, T]}|y(t)-y^h(t)| \lesssim  \varrho^g_\alpha(\mathbf X^h, \mathbf X),
\end{align}
we find that $y^h$ is close to $y$ for sufficiently small $h$. In this section, we restrict ourselves to branched rough paths $\mathbf X$ that are the limit of a lifted piece-wise linear approximation $X^h$ of $X$. This, e.g., includes semi-martingales, fractional Brownian motions with Hurst index $H>\frac{1}{4}$ and other Gaussian processes \cite{FV10}. This piece-wise linear approximation to $X$ on some grid $t_0< t_1 <\ldots <t_N =T$ is constructed as follows:
\begin{align}\label{piece-wise_linear}
 X^h(t) = X(t_k) + \frac{t-t_k}{h_k}\left[X(t_{k+1}) - X(t_k) \right],\quad t\in(t_k, t_{k+1}],
\end{align}
where $h_k=t_{k+1}-t_k$ and $k= 0, 1, \ldots, N - 1$. We assume that this piece-wise linear 
approximation converges with rate $r_0>0$, meaning that 
\begin{align*}
 \varrho^g_\alpha(\mathbf X^h, \mathbf X) = \mathcal O(h^{r_0})
\end{align*}
for sufficiently small $h$, where $h=\max_{k=0,\ldots, N-1}\vert t_{k+1}-t_k\vert$.

\begin{example}
 In \cite{FR14}, the almost sure convergence rate of $\mathbf X^h$ to $\mathbf X$ is calculated for the natural lift $\mathbf{X}$ (in the sense of \cite{FV10-2}) of a large class of Gaussian processes $X$ in the metric $\varrho_{\alpha}^g$. In particular, for the lift of a fractional Brownian motion with Hurst parameter $H \in (1/4,1,2]$, one can show that the rate $r_0$ is arbitrarily close to $2H - 1/2$ provided one chooses $\alpha$ sufficiently small.
\end{example}


Below, we analyze the order of the local error of some simplified Runge-Kutta scheme if the underlying driver is $\mathbf{X}^h$, considered as $\alpha$-branched rough path. This scheme is obtained by setting \begin{align}\label{ansatz_simple_scheme}
Z^{(k)}_{ij} = a_{ij} X^k_{t_n, t_{n+1}}\quad \text{and} \quad z^{(k)}_i =b_i X^k_{t_n, t_{n+1}}   \end{align}
in (\ref{eqn:RK_equations}), where $X^k_{t_n, t_{n+1}}$ denotes the increment of the $k$th component of $X$ on $[t_n, t_{n+1}]$.
Method \eqref{eqn:RK_equations} then becomes 
\begin{align}\label{eqn:RK_equations_simple}
 \begin{split}
      Y_i^h &= y_n^h + \sum_{k = 1}^d \sum_{j=1}^s a_{ij} f_k(Y^h_j)  X^k_{t_n, t_{n+1}} = y_n^h + \sum_{j=1}^s a_{ij} f(Y^h_j)  X_{t_n, t_{n+1}} \\
      y_{n+1}^h &= y_n^h + \sum_{k = 1}^d \sum_{i=1}^s b_i f_k(Y^h_i)  X^k_{t_n, t_{n+1}} = y_n^h + \sum_{i=1}^s b_i f(Y^h_i)  X_{t_n, t_{n+1}},
 \end{split}
\end{align}
where $\mathcal A = (a_{ij})$ is a deterministic matrix and $b=(b_i)$ a deterministic vector. This Runge-Kutta method based on the increments of $X$ was considered in \cite{HHW18} in the context 
of implicit schemes for equations driven by a certain class of Gaussian processes. We aim to find general conditions on the coefficients $b$ and $\mathcal A$ that guarantee the desired order of the
local error when approximating (\ref{eqn:ODE}). We begin with a result characterizing the branched rough path if the underlying path is given by (\ref{piece-wise_linear}). 
 
\begin{proposition}\label{branched_rp_WZ}
Let $\tau\in \mathcal T$ be a tree of order $p$, i.e, $|\tau| = p$. Then, for the branched rough path associated to the piece-wise linear approximation in (\ref{piece-wise_linear}), we have \begin{align*}
   \langle \mathbf{X}^h_{t_k, t_{k+1}}, \tau \rangle = \frac{1}{\gamma(\tau)} X^{i_1}_{t_k, t_{k+1}}   X^{i_2}_{t_k, t_{k+1}}  \dots X^{i_p}_{t_k, t_{k+1}}        
                                                                                  \end{align*}
where the $i_\ell\in \{1, 2, \ldots, d\}$ are the labels of the tree $\tau$ and $X^i_{t_k, t_{k+1}}$ is the increment of the $i$th component of $X$ on $[t_k, t_{k+1}]$.
                                                                                  \end{proposition}
\begin{proof}
We prove by induction on the height of $\tau$ that \begin{align}\label{idtoprove}
\langle \mathbf{X}^h_{t_k, t}, \tau \rangle = \frac{1}{\gamma(\tau)} \left(\frac{t-t_k}{h_k}\right)^p X^{i_1}_{t_k, t_{k+1}}   X^{i_2}_{t_k, t_{k+1}}  \dots X^{i_p}_{t_k, t_{k+1}} 
                                                   \end{align}
for $t\in (t_k, t_{k+1}]$. Setting $t=t_{k+1}$ then yields the claim. The identity is true for $\tau = 1$ and $\tau = \bullet_{i_1}$.
Let us assume that (\ref{idtoprove}) is true for all sub-trees of $\tau = [\tau_1, \ldots, \tau_n]_{i_p}$. Then, according to 
Example \ref{example_smooth}, we have \begin{align*}
&\langle \mathbf{X}^h_{t_k, t}, \tau \rangle  = \int_{t_k}^t \langle \mathbf{X}^h_{t_k,u}, \tau_1 \rangle \cdots \langle \mathbf{X}^h_{t_k,u}, \tau_n \rangle\, dX^{h, i_p}(u) 
= \int_{t_k}^t \langle \mathbf{X}^h_{t_k,u}, \tau_1 \rangle \cdots \langle \mathbf{X}^h_{t_k,u}, \tau_n \rangle \frac{X^{i_p}_{t_k, t_{k+1}}}{h_k}\, du\\
&= \int_{t_k}^t \frac{1}{\gamma(\tau_1)} \left(\frac{u-t_k}{h_k}\right)^{p_1} X^{\tilde i_1}_{t_k, t_{k+1}}  \cdots X^{\tilde i_{p_1}}_{t_k, t_{k+1}}  \cdots 
\frac{1}{\gamma(\tau_n)} \left(\frac{u-t_k}{h_k}\right)^{p_n} X^{\bar i_1}_{t_k, t_{k+1}} \cdots X^{\bar i_{p_n}}_{t_k, t_{k+1}}  \frac{X^{i_p}_{t_k, t_{k+1}}}{h_k}\, du,
                                      \end{align*}
where $p_i:= |\tau_i|$ $(i=1, \ldots, n)$. Since $\sum_{i=1}^n p_i= p-1$ and $\Pi_{i=1}^n \frac{1}{\gamma(\tau_i)} = \frac{p}{\gamma(\tau)}$, we obtain
\begin{align*}
\langle \mathbf{X}^h_{t_k, t}, \tau \rangle  = 
\int_{t_k}^t \frac{p}{\gamma(\tau)} \left(\frac{u-t_k}{h_k}\right)^{p-1} \frac{1}{h_k} X^{i_1}_{t_k, t_{k+1}}  \cdots X^{i_{p}}_{t_k, t_{k+1}}\, du 
= \frac{1}{\gamma(\tau)} \left(\frac{t-t_k}{h_k}\right)^{p} X^{i_1}_{t_k, t_{k+1}}  \cdots X^{i_{p}}_{t_k, t_{k+1}}
\end{align*}
which concludes the proof of this proposition.
\end{proof}
The local error of the simplified Runge-Kutta scheme applied to (\ref{eqn:ODE}) is defined as  $le^h(t_0, y_0; h) : = y^h(t_0 + h) - y^h_1$. We can now rewrite (\ref{rate_local}) and (\ref{cond_loc_rate})
using Proposition \ref{branched_rp_WZ}. Moreover, within the series representation given in Theorem \ref{B-series_RK}, $a(\tau)(h)$ is replaced by $a^h(\tau)(h)$ if the simplifying ansatz
(\ref{ansatz_simple_scheme}) is used. Now, the order of the local error of (\ref{eqn:RK_equations_simple}) is
 \begin{align}\label{rate_local_WZ}
        |le^h(t_0, y_0; h)| = \mathcal{O}(h^{(p+1)\alpha})                                   
                                          \end{align}
for sufficiently small $h>0$ if and only if 
\begin{align}\label{cond_loc_rate_WZ}
 \frac{1}{\gamma(\tau)} X^{i_1}_{t_0, t_{0}+h}   X^{i_2}_{t_0, t_{0}+h}  \dots X^{i_{|\tau|}}_{t_0, t_{0}+h}   = a^h(\tau)(h)\quad \forall \tau \in \mathcal{T} \text{ with } |\tau|\leq p,
\end{align}
where $\alpha$ is the H\"older regularity of $X$. Based on (\ref{cond_loc_rate_WZ}), we aim to find proper choices of $\mathcal A$ and $b$ in \eqref{eqn:RK_equations_simple} that provide
the desired local rate in (\ref{rate_local_WZ}). In order to simplify the notation in the result below, we introduce $c_i:= \sum_{j=1}^s a_{ij}$. We now formulated conditions for the order of the local 
error associated to the simplified Runge-Kutta scheme.
\begin{theorem}\label{Thm_local_error}
The simplified Runge-Kutta method \eqref{eqn:RK_equations_simple} approximating (\ref{eqn:ODE}) has a local error of order $(p+1)\alpha$, i.e., 
\begin{align*}
        |le^h(t_0, y_0; h)| = \mathcal{O}(h^{(p+1)\alpha})                                   
                                          \end{align*}
if and only if the following conditions are satisfied for all $\ell=1, \ldots, p$:
\begin{table}[th]
\centering
\[ \begin{array}{l|l}
         \ell \\ \hline ~\\[-3mm]
         1 & \sum\limits_{i=1}^sb_i\;=\;1 \\[3mm] \hline ~\\[-3mm]
         2 & \sum\limits_{i=1}^sb_ic_i\;=\;\frac12 \\[3mm] \hline ~\\[-3mm]
         3 & \sum\limits_{i=1}^sb_ic_i^2\;=\;\frac13\;,\qquad
             \sum\limits_{i=1}^s\sum\limits_{j=1}^sb_ia_{ij}c_j\;=\;\frac16 \\[3mm] 
         \hline \end{array}\]
\caption{Algebraic conditions for the local error of the simplified Runge-Kutta method.}
\label{table1}
\end{table}
\end{theorem}
\begin{proof}
Let $i_1, i_2, i_3\in\{1, \ldots, d\}$. We start analyzing (\ref{cond_loc_rate_WZ}) for all trees of order one, i.e.,  $\tau= \bullet_{i_1}$. Using the definition of $a^h(\tau)(h)$, i.e., we plug in 
(\ref{ansatz_simple_scheme}) in the definition of $a(\tau)(h)$, we obtain \begin{align*}
a^h(\bullet_{i_1})(h) := \langle z^{(i_1)}, \Phi(1)(h) \rangle = \sum_{i=1}^s z^{(i_1)}_i = \sum_{i=1}^s b_i X^{i_1}_{t_0, t_{0}+h}.
\end{align*}
Inserting this into (\ref{cond_loc_rate_WZ}), we find \begin{align*}
                           \frac{1}{\gamma(\bullet_{i_1})} X^{i_1}_{t_0, t_{0}+h} =   \sum_{i=1}^s b_i X^{i_1}_{t_0, t_{0}+h}
                                                      \end{align*}
which is equivalent to $\sum_{i=1}^s b_i=1$. We continue with the trees of order two. These are of the form $\tau = [\bullet_{i_2}]_{i_1}$. Again, we determine $a^h(\tau)(h)$ which is 
\begin{align*}
a^h(\tau)(h) := \langle z^{(i_1)}, \Phi(\bullet_{i_2})(h) \rangle =  \langle z^{(i_1)}, Z^{(i_2)} \Phi(1)(h)\rangle =  \langle b, \mathcal A\mathbf 1_s\rangle X^{i_1}_{t_0, t_{0}+h}X^{i_2}_{t_0, t_{0}+h},
\end{align*}
using the representations in (\ref{ansatz_simple_scheme}). With this expression for $a^h(\tau)(h)$, (\ref{cond_loc_rate_WZ}) becomes \begin{align*}
 \frac{1}{2} X^{i_1}_{t_0, t_{0}+h}   X^{i_2}_{t_0, t_{0}+h}   =  \langle b, \mathcal A\mathbf 1_s\rangle X^{i_1}_{t_0, t_{0}+h}X^{i_2}_{t_0, t_{0}+h}
\end{align*}
exploiting that $\gamma(\tau) = 2$. This is equivalent to $\sum_{i=1}^s b_i c_i = \frac{1}{2}$. We conclude this proof by considering the order three trees. We start with trees of the form
$\tau = [[\bullet_{i_3}]_{i_2}]_{i_1}$. Then, $a^h(\tau)(h)$ is
\begin{align*}
  \quad a^h(\tau)(h) &= \langle z^{(i_1)}, \Phi([\bullet_{i_3}]_{i_2})(h) \rangle = \langle z^{(i_1)}, Z^{(i_2)} \Phi(\bullet_{i_3})(h) \rangle = \langle z^{(i_1)}, Z^{(i_2)} Z^{(i_3)} \mathbf 1_s\rangle\\
  &=  \langle b, \mathcal A (\mathcal A \mathbf 1_s)\rangle X^{i_1}_{t_0, t_{0}+h}X^{i_2}_{t_0, t_{0}+h} X^{i_3}_{t_0, t_{0}+h}.
\end{align*}
Moreover, we see that $\gamma(\tau) = 6$. Using the above, (\ref{cond_loc_rate_WZ}) for $\tau = [[\bullet_{i_3}]_{i_2}]_{i_1}$ is equivalent to 
\begin{align*}
\frac{1}{6} = \langle b, \mathcal A (\mathcal A \mathbf 1_s)\rangle = \sum_{i=1}^s\sum_{j=1}^s b_i a_{ij} c_j.
\end{align*}
Now, the only type of tree left is the branched tree $\tau = [\bullet_{i_2}, \bullet_{i_3}]_{i_1}$. The corresponding $a^h(\tau)(h)$ is 
\begin{align*}
  \quad a^h(\tau)(h) &= \langle z^{(i_1)}, \Phi(\bullet_{i_2})(h) \Phi(\bullet_{i_3})(h)\rangle  = \langle z^{(i_1)}, Z^{(i_2)} \mathbf 1_s Z^{(i_3)} \mathbf 1_s\rangle\\
  &=  \langle b, (\mathcal A \mathbf 1_s) (\mathcal A \mathbf 1_s)\rangle X^{i_1}_{t_0, t_{0}+h}X^{i_2}_{t_0, t_{0}+h} X^{i_3}_{t_0, t_{0}+h}.
\end{align*}
Notice that the product of two vectors is meant component-wise. For this tree, (\ref{cond_loc_rate_WZ}) therefore is equivalent to \begin{align*}
    \frac{1}{3} = \frac{1}{\gamma(\tau)} = \langle b, (\mathcal A\mathbf 1_s) (\mathcal A \mathbf 1_s)\rangle = \sum_{i=1}^s b_i c_i^2                                                                                                                               
                                                                                                                                   \end{align*}
which finally proves the claim.
                                                                                                                                   \end{proof}
                                                                                                                                   
                                                                                                                                   \begin{remark}
In fact, we can easily find algebraic conditions for any $\ell>3$ in Table \ref{table1} by considering trees $\tau\in\mathcal T$ with $|\tau|>3$ in \eqref{cond_loc_rate_WZ}. This means that we can achieve
a local rate of $(p+1)\alpha$ for the simplified Runge-Kutta method for arbitrary $p\in \mathbb N$.
\end{remark}
The conditions given in Table \ref{table1} are nothing but the consistency conditions known for the case of $f\equiv 0$ in (\ref{eqn:RDE_with_drift}), see, e.g., \cite{HairerLubichWanner}.
The consistency order is the order in the step size $h$ 
of the expression $\frac{le(t_0, x_0; h)}{h}$. If all conditions in Table \ref{table1} are fulfilled, then one has a scheme of consistency order $3$ assuming $f\equiv 0$ in (\ref{eqn:RDE_with_drift}). Such $3$rd order schemes 
are well-studied in the ordinary differential equation scenario. Below, we provide just a few examples that satisfy these conditions.
\begin{example}\label{order3_methods}
We introduce the general Butcher-scheme:\[\mathcal{BS} := \begin{array}{l|llll}
      c_1& a_{11}&a_{12}&\dots&a_{s1} \\ c_2 & a_{21} & a_{22}&\dots&a_{2s} \\ 
      \vdots & \vdots &  \vdots & \ddots&\vdots \\
      c_s & a_{s1} & a_{s2} & \cdots & a_{ss} \\
      \hline
      & b_1 & b_2 & \cdots  & b_s \end{array}.\]
\begin{itemize}
 \item[(i)] An explicit Runge-Kutta scheme satisfying all conditions in Table \ref{table1} is Heun's third-order method:
 \[\mathcal{BS} = \begin{array}{l|lll}
      0& 0 & 0 & 0 \\ 
      1/3 & 1/3 & 0 & 0 \\ 
      2/3 & 0 & 2/3 & 0 \\
      \hline
      & 1/4 & 0 & 3/4 \end{array}.\]
Hence, the iterative scheme (\ref{eqn:RK_equations_simple}) is \begin{align*}
       y_{n+1}^h = y_n^h + \frac{1}{4}[f(y_n^h ) + 3f(Y^h_3)]  X_{t_n, t_{n+1}}
 \end{align*}
where $Y_3^h$ is given by \begin{align*}
 Y_3^h & = y_n^h +  \frac{2}{3} f(Y^h_2)  X_{t_n, t_{n+1}} \quad  \text{with}\quad     Y_2^h  = y_n^h +  \frac{1}{3} f(y_n^h)  X_{t_n, t_{n+1}}.                   
                      \end{align*}
 \item[(ii)] Another explicit method fulfilling the conditions in Table \ref{table1} is Kutta's third order scheme:
  \[\mathcal{BS} = \begin{array}{l|lll}
      0& 0 & 0 & 0 \\ 
      1/2 & 1/2 & 0 & 0 \\ 
      1 & -1 & 2 & 0 \\
      \hline
      & 1/6 & 2/3 & 1/6 \end{array}.\]
 Consequently, the simplified Runge-Kutta method is 
 \begin{align*}
    y_{n+1}^h &= y_n^h + \frac{1}{6}[f(y_n^h) + 4 f(Y^h_2) +f(Y^h_3) ] X_{t_n, t_{n+1}},
 \end{align*}
 where $Y_2^h $ and $Y_3^h$ are computed by \begin{align*}
Y_2^h = y_n^h + \frac{1}{2} f(y_n^h)  X_{t_n, t_{n+1}} \quad \text{and} \quad Y_3^h = y_n^h + [- f(y_n^h) + 2 f(Y^h_2) ]  X_{t_n, t_{n+1}}.     
       \end{align*}
 \end{itemize}
 \end{example}
Notice that there is much more schemes satisfying the above conditions, e.g., \cite[Corollary 5.1]{HHW18} provide two implicit Runge-Kutta methods (for stochastic differential equations 
driven by a certain class of Gaussian processes) that satisfy the requirements in Table \ref{table1}.

\section{Global rates}\label{sec:global_rates}

\subsection{Global rate of the full Runge-Kutta scheme}

Let $y_t^{s, y_0}$ denote the solution to (\ref{eqn:RDE_without_drift}) at time $t$ starting in $y_0$ at $s$, i.e., $y_s^{s, y_0}= y_0$.
Let a numerical scheme be given as the following one step method: 
  \begin{align*}
  y_{n+1} = y_n + \Phi(y_n, \mathbf{X}_{t_n,t_{n+1}}),
 \end{align*}
 where $t_0 < t_1 < \ldots < t_N = T$ is a partition of $[t_0,T]$.
Below, we analyze the order of convergence of the numerical method \eqref{eqn:RK_equations}. The next proposition shows that we loose one order from 
the local to the global error.
\begin{proposition}\label{globale_rate}
 If there is a constant $C_1>0$ such that \begin{align}\label{assumption1}
\vert y_t^{s, y_0} - y_0 - \Phi(y_0,\mathbf{X}_{s,t})\vert \leq C_1 \vert t-s\vert^{1+r}
 \end{align}
for $\vert t-s\vert$ being sufficiently small and if
\begin{align}\label{assumption2}
\vert y_t^{s, y_0} - y_t^{s, \tilde y_0}\vert \leq C_2 \vert y_0-\tilde y_0\vert 
 \end{align}
for some constant $C_2>0$, where $y_0,\tilde y_0\in\mathbb R^m$ and $0\leq s\leq t\leq T$. Then, there is some $C>0$ such that
\begin{align*}
\max_{k=0, \ldots, N}\vert y(t_k) - y_k \vert \leq C h^r
\end{align*}
for $r>0$, where $h=\max_{k=0,\ldots, N-1}\vert t_{k+1}-t_k\vert$.
\end{proposition}
\begin{proof}
We write the global error as follows 
  \begin{align}\label{telescopesum}
y_{k} - y(t_{k}) =\sum_{j=0}^{k-1} \left(y_{t_{k}}^{t_{j+1}, y_{j+1}}- y_{t_{k}}^{t_j, y_j}\right)
\end{align}                                     
using that $y_{t_{k}}^{t_0, y_0}=y(t_{k})$ and $y_{t_{k}}^{t_{k}, y_{k}}= y_{k}$. We combine \begin{align*}
        y_{t_{k}}^{t_j, y_j}= y_{t_{k}}^{t_{j+1}, y_{t_{j+1}}^{t_j, y_j}}\quad \text{and}\quad y_{j+1}  = y_j + \Phi(y_j,\mathbf{X}_{t_j,t_{j+1}})
                \end{align*}
with (\ref{telescopesum}) which yields \begin{align*}
\vert y(t_{k}) - y_{k} \vert &=\sum_{j=0}^{k-1}\vert y_{t_{k}}^{t_{j+1}, y_j + \Phi(y_j,\mathbf{X}_{t_j,t_{j+1}})} - y_{t_{k}}^{t_{j+1}, y_{t_{j+1}}^{t_j, y_j}}\vert 
\leq C_2 \sum_{j=0}^{k-1}\vert y_j + \Phi(y_j,\mathbf{X}_{t_j,t_{j+1}}) -  y_{t_{j+1}}^{t_j, y_j}\vert\\
&\leq C_1 C_2 \sum_{j=0}^{k-1}\vert t_{j+1}-t_j \vert^{r+1}\leq C_1 C_2 h^r \sum_{j=0}^{k-1}(t_{j+1}-t_j)\leq C_1 C_2 (T-t_0) h^r 
\end{align*} 
exploiting assumptions (\ref{assumption1}) and (\ref{assumption2}). This concludes the proof of this proposition.
\end{proof}

\subsection{Global rate of the simplified Runge-Kutta scheme}

In this section, we study a particular case of $\mathbf X$ being an $\alpha$-H\"older geometric rough path, $0 < \alpha \leq 1$, that can be approximated by the lift of its piece-wise linear approximated underlying path $X$. For such driver, the simplified Runge-Kutta method (\ref{eqn:RK_equations_simple}) converges. Its order is shown in the following theorem.

\begin{theorem}\label{thm:rates_simple_rk}
Let $\mathbf X$ be an $\alpha$-H\"older geometric rough path in (\ref{eqn:RDE_without_drift}), $0 < \alpha \leq 1$,
and let its piece-wise linear approximation $X^h$ be given by (\ref{piece-wise_linear}). We assume that  the Wong-Zakai approximation converges with rate $r_0>0$, meaning that
\begin{align}\label{eqn:wong_zakai}
 \sup_{t\in [t_0, T]}|y(t)-y^h(t)| =  \mathcal O(h^{r_0})
\end{align}
for sufficiently small $h$, where $h=\max_{k=0,\ldots, N-1}\vert t_{k+1}-t_k\vert$ is the maximal step size of the underlying grid, $y$ and $y^h$ are the solutions to (\ref{eqn:RDE_without_drift}) and (\ref{eqn:ODE}), respectively. 
If all conditions in Table \ref{table1} are satisfied and the right hand side $f$ is of class $\operatorname{Lip}^{\gamma}_b$ for some $\gamma > \frac{1}{\alpha}$,
then the simplified Runge-Kutta method (\ref{eqn:RK_equations_simple} converges 
with rate $\eta=\min\{r_0, 4\alpha-1\}$ to the solution of (\ref{eqn:RDE_without_drift}), i.e., 
there is a constant $C>0$ such that \begin{align*}
\max_{k=0, \ldots, N}\vert y(t_k) - y_k^h \vert \leq C h^\eta
\end{align*}
for sufficiently small $h$.
\end{theorem}

\begin{proof}
It holds that  \begin{align*}
\vert y(t_k) - y_k^h \vert \leq \sup_{t\in [t_0, T]}|y(t)-y^h(t)| + \max_{k=0, \ldots, N}\vert y^h(t_k) - y_k^h \vert.
\end{align*}
Theorem \ref{Thm_local_error} gives us a rate of $4\alpha$ for the local error of simplified Runge-Kutta method. Proposition \ref{globale_rate} now provides that 
\begin{align*}
 \max_{k=0, \ldots, N}\vert y^h(t_k) - y_k^h \vert = \mathcal O(h^{4 \alpha - 1})
\end{align*}
if assumption (\ref{assumption2}) holds true. Let $y^{h, s, y_0}_t$ denote the solution to (\ref{eqn:ODE}) with initial time $s$ and initial state $y_0$.
Since $\mathbf X$ is $\alpha$ H\"older and since $\mathbf X^h$ is convergent and hence bounded in $h$, there is a constant $K>0$ independent of $h$ such that
\begin{align*}
\vert y_t^{h, s, y_0} - y_t^{h, s, \tilde y_0}\vert \leq K \expn^{K\vert t-s\vert^\alpha} \vert y_0-\tilde y_0\vert,
 \end{align*}
 cf. \cite{FV10}. This implies (\ref{assumption2}) and concludes the proof.
\end{proof}

\begin{remark}
\begin{itemize}
 \item[(i)] From \eqref{eqn:local_hoelder_rp}, a sufficient condition for \eqref{eqn:wong_zakai} is that $\varrho_{\alpha'}^g(\mathbf X^h,\mathbf X) = \mathcal O(h^{r_0})$ for some $0 < \alpha' \leq \alpha$ in which case one has to assume $f \in \operatorname{Lip}^{\gamma}_b$ for some $\gamma > \frac{1}{\alpha'}$.
 \item[(ii)] Theorem \ref{thm:rates_simple_rk} is formulated for any roughness parameter $\alpha > 0$. For a fractional Brownian motion with Hurst parameter $H \in (1/4,1)$, 
 it gives an optimal rate of convergence in the case when $H \in (1/4,1/2]$. Indeed, from \cite{FR14}, we know that $r_0$ can be chosen arbitrarily close to $2H - 1/2$. Since $2H - 1/2 < 4H -1$,
 the convergence rate of the simplified Runge-Kutta scheme is arbitrarily close to $2H - 1/2$. This rate is the same as for the simplified Milstein scheme introduced in \cite{DNT12}, cf. \cite{FR14},
 which is believed to be optimal due to the results obtained in \cite{NTU10}. 
\end{itemize}

\end{remark}

\section{Numerical experiments}\label{sec:numerics}

We illustrate the rate of convergence of a scheme presented in Example \ref{order3_methods}. In particular, we apply Heun's method to (\ref{eqn:RDE_without_drift}) with $f_1(y) = \cos(y)$, $f_2(y)=\sin(y)$ and $y(t)\in\mathbb R$.
Then, we have \begin{align}\label{eq:numerics}
 dy(t) = \cos(y(t))\, d\mathbf{X}^1(t)+\sin(y(t))\, d\mathbf{X}^2(t),\quad y(0) = 1, \quad t \in [0, T].
\end{align}
We assume that $X$ is a path of a two-dimensional fractional Brownian motion with independent components and Hurst index $H=0.4$. Moreover, $\mathbf X$ denotes its geometric lift and $T=0.25$. This example was considered by Deya, Neuenkirch and Tindel in \cite{DNT12} in the context of rates of convergence for a Milstein scheme. We use equidistant grid points, i.e., $h=\frac{T}{N}$. We determine the maximal discretization error 
\begin{align*}
 \mathcal E(h):= \max_{k=0, \ldots, N}\vert y(t_k) - y_k^h \vert
\end{align*}
for different $h$, where $y_k^h$ is the $k$th iterate of the simplified scheme in (\ref{eqn:RK_equations_simple}) with coefficients defined in Example \ref{order3_methods} (i).
There is no explicit representation for the solution to 
(\ref{eq:numerics}). Therefore, we create a reference solution based on the numerical method for a very small step size. In Figure \ref{error_plot}, the red circles show $\mathcal E(h)$ in dependence of $h$ 
for three paths $X$ of a fractional Brownian motion. The theoretical rate of convergence for the Heun method is $2H -0.5 = 0.3$. The slopes of the regression lines in blue confirm this rate up to an acceptable deviation. 
This deviation can be explained by the fluctuations that can be expected due to the underlying small rate of convergence.
    \begin{figure}[ht]
\centering
\includegraphics[width=0.32 \linewidth]{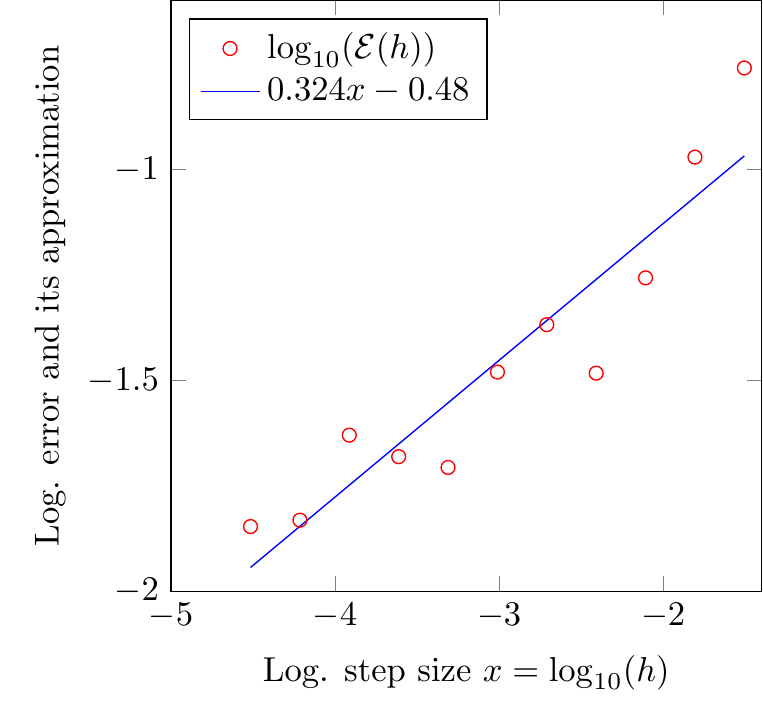}\includegraphics[width=0.32 \linewidth]{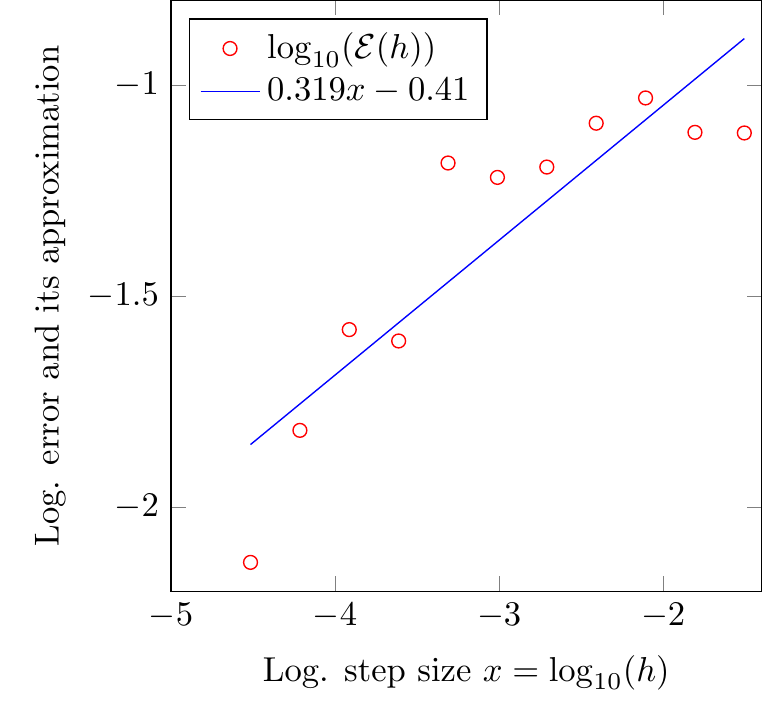}\includegraphics[width=0.32 \linewidth]{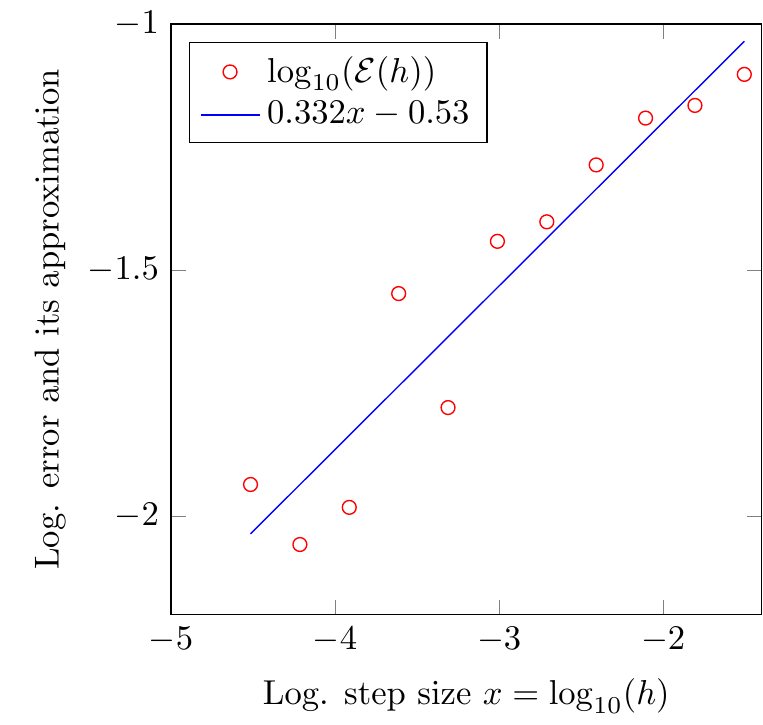}
\caption{Maximum discretization error of Heun method applied to (\ref{eq:numerics}) for three paths of fractional Brownian motion with $H=0.4$. }\label{error_plot}
\end{figure}                                                      


\subsection*{Acknowledgements}
\label{sec:acknowledgements}

SR is supported by the MATH+ project AA4-2 \textit{Optimal control in energy markets using rough analysis and deep networks}.
Work on this paper was started while MR and SR were supported by the DFG via Research Unit FOR 2402. Both authors would like to thank Rosa Prei{\ss} for
related discussions and for providing us with her Master thesis in which she corrected the proof of \cite[Proposition 3.8]{HK15}.

\bibliographystyle{alpha}
\newcommand{\etalchar}[1]{$^{#1}$}
\def\cprime{$'$}

\end{document}